\documentclass[12pt]{article}
\usepackage{latexsym,amssymb,amsmath,color,amsthm}
\usepackage{hyperref}

\theoremstyle{plain}
\newtheorem{theor}{Theorem}
\theoremstyle{remark}
\newtheorem{rem}{Remark}
\theoremstyle{plain}
\newtheorem{prop}[theor]{Proposition}
\newtheorem{cor}[theor]{Corollary}
\newtheorem{lemma}[theor]{Lemma}

\def\R{{\mathbb R}}
\def\Prob{{\mathbb P}}
\def\N{{\mathbb N}}

\def\dist{{\rm d}}
\def\Med{{\rm Med}}
\def\Exp{{\mathbb E}}
\def\Var{{\rm Var}}

\def\Vol{{\rm Vol}}

\def\Proj{{\rm Proj}}
\def\sign{{\rm sign}}
\def\spn{{\rm span}}
\def\Ell{{\mathcal E}}
\def\Id{{\rm Id}}
\def\Lip{{\rm Lip}}
\def\grad{{\rm grad}}
\def\Event{{\mathcal E}}

\title{Superconcentration, and randomized Dvoretzky's theorem for spaces with $1$-unconditional bases}

\author{Konstantin Tikhomirov\footnote{Department of Mathematics, Princeton University; email: kt12@math.princeton.edu}}

\begin{document}

\maketitle

\begin{abstract}
Let $n$ be a sufficiently large natural number and let $B$ be an origin-symmetric convex body in $\R^n$
{\it in the $\ell$-position}, and such that the space $(\R^n,\|\cdot\|_B)$
admits a $1$-unconditional basis.
Then for any $\varepsilon\in(0,1/2]$, and for
random $c\varepsilon\log n/\log\frac{1}{\varepsilon}$--dimensional subspace $E$ distributed according to the rotation-invariant
(Haar) measure, the section $B\cap E$ is 
$(1+\varepsilon)$--Euclidean with probability close to one. This shows that the ``worst-case'' dependence on
$\varepsilon$ in the randomized Dvoretzky
theorem in the $\ell$-position is significantly better than in John's position.
It is a previously unexplored feature, which has strong connections with
the concept of superconcentration introduced by S.~Chatterjee. In fact, our main result follows from the next
theorem:
Let $B$ be as before and assume additionally that $B$ has a smooth boundary and
$\Exp_{\gamma_n}\|\cdot\|_B\leq n^c\,\Exp_{\gamma_n}\big\|\grad_B(\cdot)\big\|_2$
for a small universal constant $c>0$, where $\grad_B(\cdot)$ is the gradient of $\|\cdot\|_B$
and $\gamma_n$ is the standard Gaussian measure in $\R^n$.
Then for any $p\in[1,c\log n]$ the $p$-th power of the norm $\|\cdot\|_B^p$ is $\frac{C}{\log n}$--superconcentrated
in the Gauss space.
\end{abstract}

{\small

}

\section{Introduction}

The term {\it superconcentration} was introduced by S.~Chatterjee to describe a situation when
the size of typical fluctuations of a function on a probability space is much smaller than
the bound provided by ``classical'' concentration inequalities \cite{Chatterjee}.
In this note, we are concerned with applications of the superconcentration phenomenon in asymptotic geometric analysis;
specifically, in the problem of finding large $(1+\varepsilon)$-Euclidean sections of convex bodies.
On the probabilistic level, we derive a concentration inequality for convex positively homogeneous functions in the Gauss space
satisfying some additional assumptions.
On the geometric level, we show that John's position may be a ``bad'' choice
as far as dependence of dimension on $\varepsilon$ is concerned
in the {\it randomized} Dvoretzky's theorem, and, at least for unit balls of normed spaces with a $1$-unconditional basis,
the {\it $\ell$-position} allows a substantially better bound on the dimension.

The theorem of A.~Dvoretzky \cite{Dvoretzky} asserts that for arbitrary fixed $k\in\N$ and
$\varepsilon>0$, every symmetric convex body of a large enough dimension contains
a $(1+\varepsilon)$-Euclidean $k$-dimensional section.
A proof of the theorem based on the concentration of measure was proposed by V.~Milman \cite{Milman 1971}.
In view of results of Y.~Gordon \cite{Gordon 1985} and G.~Schechtman \cite{Schechtman 1989},
who improved dependence of the dimension $k$ on $\varepsilon$,
the theorem of Milman reads: If $B$ is an origin-symmetric convex body in $\R^n$ with the Minkowski functional
$\|\cdot\|_B$ and
$$k(B):=\bigg(\frac{\Exp\|G\|_B}{\Lip(\|\cdot\|_B)}\bigg)^2$$
(where $G$ is the standard Gaussian vector in $\R^n$ and $\Lip(\|\cdot\|_B)$
is the Lipschitz constant of $\|\cdot\|_B$), then for any $\varepsilon\in(0,1]$
and any natural $k\leq c\varepsilon^2 k(B)$ the random $k$-dimensional subspace $E\subset \R^n$
uniformly distributed according to the rotation-invariant measure, cuts a $(1+\varepsilon)$-Euclidean section
$B\cap E$ with probability close to one.
The quantity $k(B)$ is often called the critical, or Dvoretzky's, dimension.
The last statement asserts that {\it most} sections of $B$
(with respect to the rotation-invariant probability measure) of the given dimension are
$(1+\varepsilon)$-Euclidean; in our note this version of Dvoretzky's theorem is called ``randomized''
(as opposed to ``existential'').
Let us note that Dvoretzky's theorem as well as numerous questions around it are covered in several monographs and surveys;
see, in particular, \cite{MS 1986, Pisier, Schechtman 2013, AGM 2015}.

The Dvoretzky--Rogers lemma implies that for any convex body $B$ in {\it John's position}
(i.e.\ such that the ellipsoid of maximal volume contained inside $B$
is the unit Euclidean ball) one has $\Exp\|G\|_{B}\geq c\sqrt{\log n}$,
whence $k(B)\geq c^2\log n$ for a universal constant $c>0$.
This yields
\begin{theor}
[{{}\bf Randomized Dvoretzky's theorem in John's position},
\cite{Milman 1971,Gordon 1985, Schechtman 1989}]\label{th: random John}
Let $B$ be an origin-symmetric convex body in $\R^n$ in John's position, and let $\varepsilon\in(0,1]$
and $k\leq c'\varepsilon^2\log n$.
Then, for random $k$-dimensional subspace $E$ uniformly distributed according to the rotation-invariant (Haar) measure, one has
$$\Prob\big\{B\cap E\mbox{ is $(1+\varepsilon)$-Euclidean}\big\}\geq 1-2n^{-c'\varepsilon^2}.$$
Here, $c'>0$ is a universal constant.
\end{theor}
In fact, in the above theorem ``$(1+\varepsilon)$-Euclidean'' can be replaced with a stronger notion
of {\it $(1+\varepsilon)$-spherical} which we define as a section
$B\cap E'$ such that $\lambda B_2^n\cap E'\subset B\cap E'\subset (1+\varepsilon)\lambda B_2^n\cap E'$,
where $\lambda$ is some positive real number and $B_2^n$ is the standard Euclidean ball in $\R^n$.

It is not difficult to show that dependence on $\varepsilon$ in the above theorem cannot be improved.
The following statement can be verified by elementary geometric arguments combined with basic probability
(for completeness, we give a proof in Section~\ref{s:john}):
\begin{prop}[Optimality of randomized Dvoretzky's theorem in John's position]\label{p: John poor concentration}
There are universal constants $c,C>0$ and $n_0\in\N$ with the following property:
For any $n\geq n_0$ there is an origin-symmetric convex body $B\subset\R^n$ in John's position
(and, moreover, the standard basis in $\R^n$ is $1$-unconditional with respect to the norm $\|\cdot\|_B$)
such that for all $\varepsilon\in(0,c]$ and $k\geq C\max(\varepsilon^2\log n,1)$,
the random $k$-dimensional subspace $E$ uniformly distributed according to the rotation-invariant (Haar) measure,
satisfies
$$\Prob\big\{B\cap E\mbox{ is $(1+\varepsilon)$-Euclidean}\big\}\leq\frac{1}{2}.$$
\end{prop}

Although Theorem~\ref{th: random John} is sharp, it raises the question whether John's position is a good
choice for generating random almost Euclidean sections, or there is another {\it canonical} position
which yields a better dependence on $\varepsilon$.
As an example, let us note that in a recent paper \cite{PV} 
dealing with Dvoretzky's theorem for subspaces of $L_p$,
G.~Paouris and P.~Valettas used a position of the unit ball other than John's.

\medskip

For an origin-symmetric convex body $B$ in $\R^n$ and a linear operator $U:\R^n\to\R^n$
define
$$\ell(B,U):=\bigg(\int \|U(x)\|_B^2\, d\gamma_n(x)\bigg)^{1/2}=\big(\Exp\, \|U(G)\|_B^2\big)^{1/2},$$
where $\|\cdot\|_B$ is the Minkowski functional of $B$ and $G$ is the standard Gaussian vector in $\R^n$.
We say that $B$ is {\it in the $\ell$-position} if $\ell(B,\Id_n)=1$ and
$$1=\det \Id_n=\sup\big\{|\det U|:\,U\in\R^{n\times n},\,\ell(B,U)\leq 1\big\}.$$
It is not difficult to check that for any body $B\subset\R^n$ and a linear operator $U_0$
with $\ell(B,U_0)= 1$
and such that
$$|\det U_0|=\sup\big\{|\det U|:\,U\in\R^{n\times n},\,\ell(B,U)\leq 1\big\},$$
the image $U_0^{-1}(B)$ is in the $\ell$-position.
The importance of the $\ell$-position in asymptotic geometric analysis was
revealed by T.~Figiel and N.~Tomczak-Jaegermann in \cite{FT}
(see also \cite[\S~12]{Tomczak}, \cite[Chapter~3]{Pisier}, \cite[Chapter~6]{AGM 2015}).

The main result of our note is the following theorem:
\begin{theor}[{{}\bf Randomized Dvoretzky's theorem in the $\ell$-position}]\label{th: random ell}
Let $B$ be an origin-symmetric convex body in $\R^n$ in the $\ell$-position,
and such that the space $(\R^n,\|\cdot\|_B)$ has a $1$-unconditional basis.
Further, let $\varepsilon\in(0,1/2]$
and $k\leq c\varepsilon\log n/\log\frac{1}{\varepsilon}$.
Then for random $k$-dimensional subspace $E\subset\R^n$ uniformly distributed
according to the rotation-invariant (Haar) measure, one has
$$\Prob\big\{B\cap E\mbox{ is $(1+\varepsilon)$--spherical}\big\}\geq 1-2n^{-c\varepsilon},$$
with the notion ``$(1+\varepsilon)$--spherical'' defined above. Here, $c>0$ is a universal constant.
\end{theor}

A version of the above statement was known in the particular case $B=[-1,1]^n$ \cite{Schechtman 2007};
see also \cite{PVZ} where the randomized Dvoretzky's theorem for $\ell_p^n$--balls is studied.
The dependence on $\varepsilon$ in Theorem~\ref{th: random ell} is sharp
in a sense that the majority of $C\varepsilon\log n/\log\frac{1}{\varepsilon}$--dimensional sections of
the standard cube are not $(1+\varepsilon)$--spherical \cite{T 2014}.
Let us emphasize once more that the above statement {\it does not} hold in general if the $\ell$-position is replaced with John's.
We conjecture that the assertion of Theorem~\ref{th: random ell} is true without the assumption that the space $(\R^n,\|\cdot\|_B)$
admits a $1$-unconditional basis.

In \cite{Schechtman 2006} G.~Schechtman
proved, by combining random and deterministic arguments with a result of N.~Alon and V.~Milman \cite{AM 1983}, that any origin-symmetric convex body $B$ contains 
a $c\varepsilon\log n/\log^2\frac{1}{\varepsilon}$--dimensional $(1+\varepsilon)$-Euclidean section.
However, in contrast with Theorem~\ref{th: random ell},
the result of \cite{Schechtman 2006} is existential in a sense that it does not provide a canonical position for a convex body in which
most of its $c\varepsilon\log n/\log^2\frac{1}{\varepsilon}$--dimensional sections
are $(1+\varepsilon)$-Euclidean.

\bigskip

Our proof of Theorem~\ref{th: random ell} is based on the following dichotomy:
Given a convex body $B$ satisfying the assumptions of the theorem, either the
expectation of the length of the gradient of the norm $\Exp_{\gamma_n}\|\grad_B(\cdot)\|_2$ is very small
compared to $\Exp_{\gamma_n}\|\cdot\|_B$
(in which case a simple analysis shows that the assertion of the theorem is true)
or the gradient is relatively ``large'' in which case we involve the superconcentration.
Of course, only the second case is of interest.

Let $n$ be a natural number and let $\gamma_n$ be the standard Gaussian measure in $\R^n$.
For any sufficiently smooth real-valued function $f$ in $\R^n$ one has
$$\Var_{\gamma_n}(f):=\int f(x)^2\, d\gamma_n(x)-\bigg(\int f(x)\, d\gamma_n(x)\bigg)^2\leq 
\int \|\grad f(x)\|_2^2\,d\gamma_n(x),$$
where $\grad f$ is the gradient of $f$ ({{}\it the Poincar\'e inequality}).
A function $f$ in $\R^n$ is called {\it $\delta$-superconcentrated} (for some $\delta<1$) if
$$\Var_{\gamma_n}(f)\leq\delta \int \|\grad f(x)\|_2^2\,d\gamma_n(x).$$
The setting of actual interest involves a sequence of functions (indexed by the dimension $n$)
such that $\delta=\delta(n)$ tends to zero with $n\to\infty$.
We refer to \cite{Chatterjee} for definition of superconcentration in a more general context,
its relation to other properties (called ``chaos'' and ``multiple valleys''),
as well as for results dealing with specific probabilistic models.
Theorem~\ref{th: random ell} of this note follows from the next result.
\begin{theor}\label{th: superconcentration}
There are universal constants $c,C>0$ and $n_0\in\N$ with the following property.
Let $n\geq n_0$ and let $B$ be an origin-symmetric convex body in $\R^n$ in the $\ell$-position,
with a smooth boundary, and such that the space $(\R^n,\|\cdot\|_B)$ admits a $1$-unconditional basis.
Further, assume that $\Exp\|G\|_B\leq n^c\,\Exp\|\grad_B(G)\|_2$, where $G$ denotes the standard Gaussian
vector in $\R^n$ and $\grad_B(\cdot)$ is the gradient of the norm $\|\cdot\|_B$.
Then for any $p\in[1,c\log n]$ the function $\|\cdot\|_B^p$ is $\frac{C}{\log n}$--superconcentrated
in the Gauss space.
\end{theor}
The main tool in the proof of Theorem~\ref{th: superconcentration} is Talagrand's $L_1-L_2$ bound
(see Theorem~\ref{th: talagrand l1l2}),
which we combine with some special properties of the $\ell$-position
(``balancing conditions'').
The proof is not difficult and admits various generalizations in a sense
that the $\ell$-position can be replaced with other transformations of the convex body
that provide appropriate ``balancing'' in regard to the Gaussian measure
(we'll return to this issue at the end of the paper).

\section{Notation and preliminaries}

Let $n$ be a natural number.
The canonical basis in $\R^n$ will be denoted by $e_1,e_2,\dots,e_n$ and the standard
inner product --- by $\langle \cdot,\cdot\rangle$.
Given a set of vectors $\{y_1,y_2,\dots,y_k\}$ in $\R^n$, we denote their linear span by $\spn\{y_1,y_2,\dots,y_k\}$.
For a subspace $E\subset\R^n$, $E^\perp$ is its orthogonal complement in $\R^n$
and $\Proj_E$ is the orthogonal projection operator onto $E$.
Given a boolean variable $b$, denote by $\chi_{b}$ the indicator function of $b$, so that
$\chi_{b}=1$ if and only if $b$ is true. Similarly, for an event $\Event$ denote by $\chi_{\Event}$
the indicator function of the event.

A convex body in $\R^n$ is any compact convex set with non-empty interior.
Everywhere in this note, we say that the boundary $\partial B$ of a convex body $B$ is {\it smooth}
if every point of $\partial B$ admits a unique tangent hyperplane.
Given an origin-symmetric convex body $B$, denote by $\|\cdot\|_B$ its Minkowski functional.
By some abuse of notation, for any subspace $E\subset\R^n$ we denote by $\|\cdot\|_{B\cap E}$
the Minkowski functional of $B\cap E$ considered as a convex body inside $E$.
Further, given a $k$-dimensional subspace $E\subset\R^n$, we say that the section $B\cap E$ is $L$--Euclidean
(for some $L\geq 1$) if the Banach--Mazur distance from $B\cap E$ to a $k$-dimensional Euclidean ball
is bounded from above by $L$.

A basis $y_1,y_2,\dots,y_n$ of an $n$-dimensional normed space $W$ with a norm $\|\cdot\|$ is {\it $1$-unconditional}
if $\|\sum_{i=1}^n a_i y_i\|=\|\sum_{i=1}^n \sigma_i a_i y_i\|$ for any scalars $a_1,a_2,\dots,a_n$
and any signs $\sigma_1,\sigma_2,\dots,\sigma_n\in\{-1,1\}$.
The canonical basis of $\R^n$ is $1$-unconditional with respect to a norm $\|\cdot\|$ if and only if the
unit ball of $\|\cdot\|$ is symmetric with respect to coordinate hyperplanes.

Given a real-valued function $f$ in $\R^n$, by $\Lip(f)$ we denote its Lipschitz constant.

The set of all $k$-dimensional subspaces of $\R^n$
admits a unique normalized rotation-invariant Borel measure (the Haar measure).
Whenever we speak about a random subspace in this note, we assume it is distributed according to that measure.
The standard Gaussian measure in $\R^n$ is denoted by $\gamma_n$, the standard Gaussian vector --- by $G$,
and the standard real-valued Gaussian variable --- by $g$ (or $g_i$'s when there are several of them).

Universal constants will be denoted by $C,c$, etc.\ and their value may be different on different ocasions.

\subsection{Gaussian concentration inequalities}

The next theorem (with a worse constant in the exponent) is due to G.~Pisier.

\begin{theor}[{see, in particular, \cite[Chapter~4]{Pisier}, \cite[p.~12]{LT} or \cite[Chapter~2]{Ledoux 2001}}]\label{th: gauss concentration}
Let $G$ be a standard Gaussian vector in $\R^N$ and $f:\R^N\to\R$ be a $1$-Lipschitz function.
Then
$$\Prob\big\{f(G)-\Exp\,f(G)\geq t\big\}\leq \exp(-t^2/2),\;\;\;t>0.$$
\end{theor}
The next statement plays a crucial role in our analysis (a version of this inequality for 
the uniform measure on discrete cube was proved by M.~Talagrand in \cite{Talagrand}):
\begin{theor}[{Talagrand's $L_1\text{--}L_2$ bound; see, in particular, \cite{CEL}, \cite[Chapter~5]{Chatterjee}}]\label{th: talagrand l1l2}
Let $f$ be an absolutely continuous function in $\R^N$ and let $\partial_i f$ ($i\leq N$) denote $i$-th component of the gradient of $f$.
Then
$$\Var\big(f(G)\big)\leq C\sum\limits_{i=1}^N \frac{\Exp|\partial_i f(G)|^2}
{1+\log\big(\sqrt{\Exp|\partial_i f(G)|^2}/\Exp|\partial_i f(G)|\big)},$$
where $C>0$ is a universal constant.
\end{theor}

\subsection{Canonical positions of convex bodies}

Given an origin-symmetric convex body $B$, its {\it position} is any convex body $T(B)$
for some invertible linear transformation $T$. The two {\it canonical} positions we consider in this
note are John's and the $\ell$-position, which were defined in the introduction.
Recall that a position $T(B)$ is John's if the ellipsoid of maximal volume contained in $T(B)$ is the unit Euclidean ball.
The concept was used by F.~John \cite{John} to estimate the Banach--Mazur distance of arbitrary convex body to the Euclidean ball
(see also \cite{Ball}).

A viewpoint to canonical positions involving arbitrary norms on the space of linear operators $\mathcal L(\R^n,\R^n)$
was developed by D.R.~Lewis \cite{Lewis} (see also \cite[Chapter~3]{Tomczak}, \cite[Chapter~3]{Pisier},
as well as an alternative approach of A.~Giannopoulos and V.~Milman \cite{GM} based on isotropic measures on the Euclidean sphere).
In particular, the $\ell$-position can be defined as a linear transformation $T(B)$ satisfying
$$\ell(T(B),\Id_n)=\ell(B,T^{-1})=1,\quad\quad \ell^*(T(B),\Id_n)=n,$$
where $\ell(\cdot)$ is the norm on the space of linear operators in $\R^n$ defined in the introduction,
and $\ell^*(\cdot)$ is the norm {\it in trace duality} with $\ell(\cdot)$ \cite{Lewis,Tomczak,Pisier,AGM 2015}.
It is easy to see that the $\ell$-position is rotation-invariant in a sense that, together with $T(B)$,
any linear image of the form $UT(B)$ ($U\in O_n$) is in the $\ell$-position.
At the same time, the $\ell$-position is {\it unique} up to an orthogonal transformation \cite[Proposition~14.3]{Tomczak}.
In fact, the following stability result is true:
\begin{lemma}[Stability of the $\ell$-position]\label{l:stability}
Let $n>1$ and let $B$ be an origin-symmetric convex body in $\R^n$ in the $\ell$-position.
Then for any $\delta>0$ there is $\kappa=\kappa(\delta)>0$ depending only on $\delta$ with the following property:
whenever $T$ is an invertible linear operator in $\R^n$ with $\ell(T(B),\Id_n)=1$ and $|\det T|\leq 1+\kappa$, there
exists an orthogonal transformation $U$ of $\R^n$ such that $(1-\delta)B\subset UT(B)\subset (1+\delta)B$.
\end{lemma}
\begin{proof}
The proof to a large extent follows the argument in \cite[Proposition~14.3]{Tomczak}.
Fix a small $\delta>0$ and define $\kappa:=\frac{\delta^2}{4+4\delta}$.
Now, let $T$ be an operator satisfying the assumptions of the lemma.
Choose an orthogonal operator $U$
so that $T=U^{-1}P$, with $P$ being positive definite ({\it{}the polar decomposition} of $T$).
Then $\det P\leq 1+\kappa$ and $\ell(P(B),\Id_n)=1$. It remains to show that
$1-\delta\leq\lambda_{\min}(P)\leq \lambda_{\max}(P)\leq 1+\delta$.
We shall prove this by contradiction. Assume that either $\lambda_{\min}(P)<1-\delta$
or $\lambda_{\max}(P)>1+\delta$. Define an operator $W$ via its inverse:
$W^{-1}:=\frac{1}{2}\big(P^{-1}+\Id_n\big)$.
Clearly, if $\lambda_1,\lambda_2,\dots,\lambda_n$ are eigenvalues of $P$ then
$$\det W^{-1}=\prod_{i=1}^n\frac{{\lambda_i}^{-1}+1}{2}.$$
Obviously, $\frac{{\lambda_i}^{-1}+1}{2}\geq {\lambda_i}^{-1/2}$ for all $i\leq n$ and,
additionally, as at least one of the eigenvalues $\lambda_{i_0}$ satisfies $|\lambda_{i_0}-1|>\delta$, we have
\begin{align*}
\frac{{\lambda_{i_0}}^{-1}+1}{2}
&={\lambda_{i_0}}^{-1/2}\,
\frac{{\lambda_{i_0}}^{-1/2}+{\lambda_{i_0}}^{1/2}}{2}\\
&> {\lambda_{i_0}}^{-1/2}\,\frac{{(1+\delta)}^{-1/2}+{(1+\delta)}^{1/2}}{2}\\
&= {\lambda_{i_0}}^{-1/2}\,\sqrt{1+\frac{\delta^2}{4+4\delta}}\\
&={\lambda_{i_0}}^{-1/2}\,\sqrt{1+\kappa}.
\end{align*}
Thus,
$$\det W^{-1}>\sqrt{1+\kappa}\prod_{i=1}^n {\lambda_i}^{-1/2}\geq 1,$$
whence $\det W<1$. Next, observe that for any vector $x\in\R^n$ we have
$$\|x\|_{W(B)}^2=\|W^{-1}x\|_B^2\leq \bigg(\frac{\|P^{-1}x\|_B+\|x\|_B}{2}\bigg)^2
\leq \frac{1}{2}\|x\|_{P(B)}^2+\frac{1}{2}\|x\|_{B}^2.$$
Thus, $\ell(W(B),\Id_n)^2\leq \frac{1}{2}\ell(P(B),\Id_n)^2+\frac{1}{2}\ell(B,\Id_n)^2=1$ while $|\det W|<1$.
This contradicts the assumption that $B$ is in the $\ell$-position.
\end{proof}
As a simple corollary, we obtain
\begin{cor}\label{c:smooth}
Let $n>1$ and let $B$ be an origin-symmetric convex body in the $\ell$-position.
Then for any $\delta>0$ there is an origin-symmetric convex body $B_\delta$ with a smooth boundary,
in the $\ell$-position, and such that $(1-\delta)B\subset B_\delta\subset (1+\delta)B$.
Moreover, if the norm $\|\cdot\|_B$ admits a $1$-unconditional basis then $B_\delta$ can be defined so that
$\|\cdot\|_{B_\delta}$ is $1$-unconditional as well.
\end{cor}
\begin{proof}
Fix the convex body $B$ and a small positive $\delta$.
Define a positive number $\widetilde \kappa:=\min(\kappa(\delta/4),\delta/4)$,
where the function $\kappa$ is taken from Lemma~\ref{l:stability}.
First, one can construct a smooth approximation $B'$
of $B$ satisfying the relations
$$(1+\widetilde\kappa/4)^{-1/n}B'\subset B\subset (1+\widetilde\kappa/4)^{1/n}B',$$
and such that $\|\cdot\|_{B'}$ is $1$-unconditional whenever $\|\cdot\|_B$ is
(see, for example, \cite{Koldobsky}).
The inclusion relations imply that $(1+\widetilde\kappa/4)^{-1/n}\leq\ell(B',\Id_n)\leq (1+\widetilde\kappa/4)^{1/n}$,
whence, applying an appropriate dilation, we get a smooth convex body $B''$ satisfying
\begin{equation}\label{eq:aux659}
(1+\widetilde\kappa)^{-1/n}B\subset B''\subset (1+\widetilde\kappa)^{1/n}B
\end{equation}
and such that $\ell(B'',\Id_n)=1$. Now, let $T$ be an invertible linear transformation
so that $T(B'')$ is in the $\ell$-position.
Obviously, $T(B'')\subset (1+\widetilde\kappa)^{1/n} T(B)$, and, as $B$ is in the $\ell$-position,
we have $\ell(T(B),\Id_n)\geq |\det T|^{-1/n}$. Thus,
$1=\ell(T(B''),\Id_n)\geq (1+\widetilde\kappa)^{-1/n}|\det T|^{-1/n}$, i.e.\ $|\det T^{-1}|\leq 1+\widetilde\kappa$.
By Lemma~\ref{l:stability} (applied to $T(B'')$ and operator $T^{-1}$), there is an orthogonal transformation $U$
such that
$$(1-\delta/4)T(B'')\subset UT^{-1}T(B'')=U(B'')\subset (1+\delta/4)T(B''),$$
whence
$$(1-\delta/4)U^{-1}T(B'')\subset B''\subset (1+\delta/4)U^{-1}T(B'').$$
Together with \eqref{eq:aux659}, this implies that $U^{-1}T(B'')$ is the smooth convex body in the $\ell$-position satisfying
the required conditions.
\end{proof}
\begin{rem}
Corollary~\ref{c:smooth} will allow us to reduce the proof of Theorem~\ref{th: random ell}
to the case when the underlying convex body is smooth.
\end{rem}

The next statement is intuitively obvious; we give its proof for completeness.
\begin{lemma}\label{l: uncond}
Let $B$ be an origin-symmetric convex body in $\R^n$ in the $\ell$-position, and assume that
the normed space $(\R^n,\|\cdot\|_B)$ admits a $1$-unconditional basis. Then the basis is orthogonal
with respect to the canonical inner product in $\R^n$. 
\end{lemma}
\begin{proof}
Let $x_1,x_2,\dots,x_n$ be a $1$-unconditional basis in $(\R^n,\|\cdot\|_B)$,
and suppose that it is not orthogonal. Without loss of generality, we can assume that $H:=\spn\{x_1,\dots,x_{n-1}\}$
and $x_n$ are not orthogonal. Let $T$ be the linear transformation of $\R^n$ given by its action on the basis vectors:
$Tx_i=x_i$ for all $i\leq n-1$, and $Tx_n=x_n-\Proj_{H}x_n$, where $\Proj_H$ is the orthogonal projection onto $H$.
It is easy to see that the transformation $T$ is volume-preserving. Further,
define a convex body $B'$ via its Minkowski functional:
$$\Big\|\sum_{i=1}^n a_i T x_i\Big\|_{B'}:=\Big\|\sum_{i=1}^n a_i x_i\Big\|_{B},\quad\mbox{for all }a_i\in\R,\;\;i\leq n.$$
Thus, $B'$ is a (non-orthogonal) linear transformation of $B$ and $\Vol(B)=\Vol(B')$.
We will show that $\ell(B',\Id_n)\leq\ell(B,\Id_n)$ which, in view of the uniqueness of the $\ell$-position
mentioned above (see \cite[Proposition~14.3]{Tomczak} or the last corollary) leads to contradiction.

Let $G'$ be the standard $(n-1)$--dimensional Gaussian vector in $H$ and let $g_n$ be the standard Gaussian variable
independent from $G'$. We consider three random variables $\xi,\eta_1,\eta_2$ on the probability space given by
\begin{align*}
\xi&:=\Big\|G'+\frac{Tx_n}{\|Tx_n\|_2}g_n\Big\|_{B'};\\
\eta_1&:=\Big\|G'+\frac{Tx_n}{\|Tx_n\|_2}g_n\Big\|_{B};\\
\eta_2&:=\Big\|G'-\frac{Tx_n}{\|Tx_n\|_2}g_n\Big\|_{B}.
\end{align*}
Obviously, $\ell(B,\Id_n)=\big(\Exp\,\eta_1^2\big)^{1/2}=\big(\Exp\,\eta_2^2\big)^{1/2}$, and
$\ell(B',\Id_n)=\big(\Exp\,\xi^2\big)^{1/2}$. At the same time,
using $1$-unconditionality of the basis $\{x_1,\dots,x_n\}$ with respect to $\|\cdot\|_B$, we obtain
\begin{align*}
\eta_1+\eta_2&=\Big\|G'+\frac{Tx_n}{\|Tx_n\|_2}g_n\Big\|_{B}+\Big\|G'-\frac{Tx_n}{\|Tx_n\|_2}g_n\Big\|_{B}\\
&=\Big\|G'+\frac{x_n-\Proj_H x_n}{\|Tx_n\|_2}g_n\Big\|_{B}+\Big\|G'-\frac{x_n-\Proj_H x_n}{\|Tx_n\|_2}g_n\Big\|_{B}\\
&=\Big\|G'+\frac{x_n-\Proj_H x_n}{\|Tx_n\|_2}g_n\Big\|_{B}+\Big\|G'+\frac{x_n+\Proj_H x_n}{\|Tx_n\|_2}g_n\Big\|_{B}\\
&\geq 2\Big\|G'+\frac{x_n}{\|Tx_n\|_2}g_n\Big\|_{B}\\
&=2\Big\|G'+\frac{Tx_n}{\|Tx_n\|_2}g_n\Big\|_{B'}\\
&=2\xi.
\end{align*}
Thus, by the triangle inequality we get
$$2\big(\Exp\,\xi^2\big)^{1/2}\leq \big(\Exp\,\eta_1^2\big)^{1/2}+\big(\Exp\,\eta_2^2\big)^{1/2},$$
whence
$$\ell(B',\Id_n)\leq\ell(B,\Id_n).$$
This implies that $B'$ must also be in the $\ell$-position contradicting the fact that
the position is unique up to an orthogonal transformation.
\end{proof}

\subsection{The gradient}

Let $B$ be an origin-symmetric convex body in $\R^n$ with a smooth boundary.
For any point $x\in\R^n\setminus \{0\}$,
the gradient $\grad_B(x)$ of the function $\|\cdot\|_B$ at point $x$ is well defined.
It is not difficult to check that
\begin{equation}\label{eq: grad max property}
\|x\|_B=\langle \grad_B(x),x\rangle=\sup_{y\in \R^n\setminus\{0\}}\langle \grad_B(y),x\rangle,
\end{equation}
and that $\grad_B(\lambda x)=\grad_B(x)=-\grad_B(-x)$ for all $x\in\R^n\setminus\{0\}$ and $\lambda>0$.
Further, the gradient of $\|\cdot\|_B$ is continuous everywhere on its domain.

The next statement follows from the fact that any $1$-unconditional norm in $\R^n$ is a monotone function
in the positive cone, as well as from hyperplane symmetries. We omit the proof.
\begin{lemma}\label{l: unconditional grad}
Let $B$ be a convex body in $\R^n$ with a smooth boundary such that the standard basis $e_1,e_2,\dots,e_n$
is $1$-unconditional with respect to $\|\cdot\|_B$.
Then for every point $(x_1,x_2,\dots,x_n)\in\R^n\setminus\{0\}$
and every collection of signs $(\sigma_j)_{j=1}^n\in\{-1,1\}^n$ we have
$$0\leq \Big\langle\grad_B\Big(\sum\nolimits_{j\leq n}x_j e_j\Big),x_ie_i\Big\rangle
=\Big\langle\grad_B\Big(\sum\nolimits_{j\leq n}\sigma_j x_j e_j\Big),\sigma_i x_i e_i\Big\rangle,\quad i\leq n.$$
\end{lemma}
As an almost immediate consequence, we obtain
\begin{lemma}\label{l: l1 grad}
Let $B$ be a convex body in $\R^n$ with a smooth boundary such that the standard basis $e_1,e_2,\dots,e_n$
is $1$-unconditional with respect to $\|\cdot\|_B$.
Then for any $p\geq 1$ we have
$$
\Exp\|\grad_B(G)\|_1^p=\Exp\Big(\sum_{i=1}^n|\langle \grad_B(G),e_i\rangle|\Big)^p
\leq \big(\pi/2\big)^{p/2}\,\Exp\|G\|_B^p.
$$
\end{lemma}
\begin{proof}
Let $G'$ be an independent copy of $G$. Then, appying Lemma~\ref{l: unconditional grad} and
formula \eqref{eq: grad max property}, as well as standard estimates for the moments of Gaussians, we get
\begin{align*}
\Exp\Big(\sum_{i=1}^n|\langle \grad_B(G),e_i\rangle|\Big)^p
&\leq\big(\pi/2\big)^{p/2}\,\Exp\Big(\sum_{i=1}^n|\langle \grad_B(G),e_i\rangle\langle G',e_i\rangle|\Big)^p\\
&=\big(\pi/2\big)^{p/2}\,
\Exp\Big(\sum_{i=1}^n\langle \grad_B(G),e_i\rangle\sign(\langle G,e_i\rangle\langle G',e_i\rangle)\langle G',e_i\rangle\Big)^p\\
&\leq \big(\pi/2\big)^{p/2}\,
\Exp\Big(\sum_{i=1}^n\langle \grad_B(Z),e_i\rangle\sign(\langle G,e_i\rangle\langle G',e_i\rangle)\langle G',e_i\rangle\Big)^p\\
&=\big(\pi/2\big)^{p/2}\,\Exp\|Z\|_B^p,
\end{align*}
where $Z=\sum_{i=1}^n \sign(\langle G,e_i\rangle\langle G',e_i\rangle)\langle G',e_i\rangle e_i$.
It remains to note that $Z$ is the standard Gaussian vector in $\R^n$.
\end{proof}

Let us state one more simple geometric property of the gradient:
\begin{lemma}\label{l: grad is monotone}
Let $B$ be a smooth convex body in $\R^n$ such that the standard basis $e_1,e_2,\dots,e_n$
is $1$-unconditional with respect to $\|\cdot\|_B$.
Then for any $i\leq n$ and any fixed numbers $x_j$ ($j\neq i$),
the function $|\langle \grad_B(x_1,x_2,\dots,x_n),e_i\rangle|$ of one variable $x_i\in\R$ is non-increasing on $(-\infty,0)$
and non-decreasing on $(0,\infty)$. 
\end{lemma}

The next elementary observation follows directly from property~\eqref{eq: grad max property}.
\begin{lemma}\label{l: grad trivial upper bound}
Let $B$ be an origin-symmetric convex body in $\R^n$ with a smooth boundary.
There is a universal constant $C>0$ such that for all $x\in\R^n\setminus\{0\}$ we have
$$\|\grad_B(x)\|_2\leq C\Exp\|G\|_B.$$
\end{lemma}

\section{Basic properties of the $\ell$-position}

Let us note that all auxiliary results 
proved in this section work for arbitrary origin-symmetric
convex bodies with smooth boundaries in the $\ell$-position.

\bigskip

\begin{lemma}\label{l: ell equations}
Let $B\subset\R^n$ be a smooth origin-symmetric convex body in the $\ell$-position. Then
$$\Exp\big(\|G\|_B \langle \grad_B(G),u\rangle \langle G,u\rangle\big)=\frac{1}{n}\Exp\|G\|_B^2=\frac{1}{n}\quad
\mbox{ for any }\quad u\in S^{n-1}.$$
\end{lemma}
\begin{proof}
We will show that
$$\Exp\big(\|G\|_B  \langle\grad_B(G),e_i\rangle \langle G,e_i\rangle\big)=\frac{1}{n}\Exp\|G\|_B^2=\frac{1}{n},\quad i\leq n;$$
the statement will then follow by rotation-invariance of the $\ell$-position.
Fix for a moment any $i\leq n$, take $\varepsilon\in(0,1)$ and define a diagonal operator $D=D_\varepsilon$ via its diagonal entries:
$$d_{jj}=\begin{cases}(1-\varepsilon)^{n-1},&\mbox{ if }j=i;\\(1-\varepsilon)^{-1},&\mbox{ otherwise.}\end{cases}$$
We clearly have $\det D=1$, and, in view of \eqref{eq: grad max property},
\begin{align*}
\Exp \|DG\|_B^2&=\Exp\bigg(\sum_{j=1}^n \langle\grad_B(DG),e_j\rangle d_{jj} \langle G,e_j\rangle \bigg)^2\\
&=\Exp\bigg(\sum_{j=1}^n \langle\grad_B(DG),e_j\rangle \langle G,e_j\rangle
+\sum_{j=1}^n \langle\grad_B(DG),e_j\rangle (d_{jj}-1) \langle G,e_j\rangle \bigg)^2\\
&\leq\Exp\bigg(\sum_{j=1}^n \langle\grad_B(G),e_j\rangle \langle G,e_j\rangle
+\sum_{j=1}^n \langle\grad_B(DG),e_j\rangle (d_{jj}-1) \langle G,e_j\rangle \bigg)^2\\
&=\Exp\bigg(\|G\|_B-\varepsilon(n-1) \langle\grad_B(DG),e_i\rangle\langle G,e_i\rangle\\
&\hspace{1cm}+\varepsilon\sum_{j\neq i} \langle\grad_B(DG),e_j\rangle \langle G,e_j\rangle
+o(\varepsilon)\|G\|_B\bigg)^2\\
&=\Exp\bigg(\|G\|_B-\varepsilon n \langle\grad_B(DG),e_i\rangle\langle G,e_i\rangle\\
&\hspace{1cm}+\varepsilon\sum_{j=1}^n \langle\grad_B(DG),e_j\rangle \langle G,e_j\rangle
+o(\varepsilon)\|G\|_B\bigg)^2\\
&=\Exp\bigg(\|G\|_B+\varepsilon\sum_{j=1}^n \langle\grad_B(DG),e_j\rangle \langle G,e_j\rangle\bigg)^2\\
&\hspace{1cm}-2\varepsilon n\Exp\big(\|G\|_B \langle\grad_B(DG),e_i\rangle\langle G,e_i\rangle\big)+o(\varepsilon).
\end{align*}
Further, as $\grad_B(\cdot)$ is continuous at every point of $\R^n\setminus\{0\}$, we get
\begin{align*}
\Exp&\bigg(\|G\|_B+\varepsilon\sum_{j=1}^n \langle\grad_B(DG),e_j\rangle \langle G,e_j\rangle\bigg)^2
-2\varepsilon n\Exp\big(\|G\|_B \langle\grad_B(DG),e_i\rangle\langle G,e_i\rangle\big)\\
&\leq
\Exp\bigg(\|G\|_B+\varepsilon\sum_{j=1}^n \langle\grad_B(G),e_j\rangle \langle G,e_j\rangle\bigg)^2
-2\varepsilon n\Exp\big(\|G\|_B\langle\grad_B(G),e_i\rangle\langle G,e_i\rangle\big)+o(\varepsilon)\\
&=(1+2\varepsilon)\Exp\|G\|_B^2-2\varepsilon n\Exp\big(\|G\|_B\langle\grad_B(G),e_i\rangle\langle G,e_i\rangle\big)+o(\varepsilon).
\end{align*}
On the other hand, in view of the definition of the $\ell$-position, we have $\Exp \|G\|_B^2\leq \Exp\|DG\|_B^2$
for any $\varepsilon$. Combining this with the above inequalities, we obtain
$$\Exp \|G\|_B^2\leq (1+2\varepsilon)\Exp\|G\|_B^2-2\varepsilon n\Exp\big(\|G\|_B\langle\grad_B(G),e_i\rangle\langle G,e_i\rangle\big)+o(\varepsilon).$$
Taking the limit when $\varepsilon\to 0$, we get
$$ n\Exp\big(\|G\|_B\langle\grad_B(G),e_i\rangle\langle G,e_i\rangle\big)\leq \Exp\|G\|_B^2,\quad\quad i\leq n.$$
At the same time, obviously
$$\sum_{i=1}^n \Exp\big(\|G\|_B \langle\grad_B(G),e_i\rangle\langle G,e_i\rangle\big)=\Exp\|G\|_B^2.$$
Thus, the above relations must be equalities for all $i$.
\end{proof}

\begin{lemma}\label{l: expectation}
For any $\delta>0$ and $p\in[1,\infty)$ there are numbers $n_0=n_0(\delta)$ depending only on $\delta$
and $c_{\delta,p}>0$ depending on $\delta$ and $p$ with the following property.
Let $n\geq n_0$, and let $B$ be an origin-symmetric convex body in $\R^n$ such that
\begin{equation}\label{eq: aux 349}
\Prob\big\{\|g x\|_B^p\geq \Exp\|G\|_B^p\big\}\leq n^{-\delta}\quad\mbox{ for any vector }x\in S^{n-1}.
\end{equation}
Then
$$\frac{\Exp \|G\|_B}{\Lip(\|\cdot\|_B)}\geq c_{\delta,p}\sqrt{\log n}.$$
\end{lemma}
\begin{proof}
Without loss of generality, $n^\delta$ is large and $\Lip(\|\cdot\|_B)=1$.
Fix a vector $x\in S^{n-1}$ with $\|x\|_B=1$. Standard deviation estimates for Gaussian variables imply
$$\Prob\big\{\|gx\|_B\geq \sqrt{\log n^{\delta}}\big\}\geq n^{-\delta},$$
whence, in view of \eqref{eq: aux 349}, we have
$$\Exp\|G\|_B^p\geq \big(\log n^{\delta}\big)^{p/2}.$$
It remains to note that $\Exp\|G\|_B\geq c_p\big(\Exp\|G\|_B^p\big)^{1/p}$
for some $c_p>0$ depending only on $p$ (see, for example, \cite[Corollary~3.2]{LT}).
\end{proof}

Together Lemmas~\ref{l: ell equations} and~\ref{l: expectation} imply

\begin{prop}\label{p: ell lipschitz}
There are universal constants $n_o\in\N$ and $c>0$ with the following property.
Let $n\geq n_0$, and let $B$ be an origin-symmetric convex body in $\R^n$ with a smooth boundary in the $\ell$-position.
Then
$$\frac{\Exp \|G\|_B}{\Lip(\|\cdot\|_B)}\geq c\sqrt{\log n}.$$
\end{prop}
\begin{proof}
We can assume that $n$ is large.
We will construct an orthogonal basis $y_1,y_2,\dots,y_n$ in $\R^n$ as follows.
First, there is a vector $y_1$ with $\|y_1\|_2=\Lip(\|\cdot\|_B)$
such that $B\subset \big\{x\in\R^n:\,|\langle x,y_1\rangle|\leq 1\big\}$.
We set $H_1:={y_1}^\perp$. Now, assuming that $y_1,y_2,\dots,y_k$ are constructed, choose a vector
$y_{k+1}\in H_k:=\spn\{y_1,y_2,\dots,y_k\}^\perp$ with
$\|y_{k+1}\|_2=\Lip(\|\cdot\|_{B\cap H_k})$ such that
$B\cap H_k\subset \big\{x\in H_k:\,|\langle x,y_{k+1}\rangle|\leq 1\big\}$.

Note that $\|y_k\|_B=\|y_k\|_2^2$ for any $k\leq n$ and that $\|y_{k+1}\|_2\leq \|y_{k}\|_2$ for all $k\leq n-1$.
Now, set $q:=\Exp \|G\|_B$, $m:=\lfloor \sqrt{n}\rfloor$, and consider two cases.

\begin{itemize}

\item Suppose that $\|y_{m}\|_2\geq \frac{\sqrt{q\,\Lip(\|\cdot\|_B)}}{(\log n)^{1/4}}$.
In view of the triangle inequality and the definition of $y_i$'s, we have
$$\Exp\|G\|_B\geq\Exp\Big\|\sum_{i=1}^m \frac{y_i}{\|y_i\|_2}g_i\Big\|_B
\geq \frac{1}{2}\Exp\max_{i\leq m}\frac{|g_i|\|y_i\|_B}{\|y_i\|_2}
=\frac{1}{2}\Exp\max_{i\leq m}|g_i|\|y_i\|_2$$
(here and further in the proof $g_i$'s are independent standard Gaussians).
Then standard estimates for the maximum of independent Gaussian variables \cite[p.~302]{DN 2003},
together with the assumption on the Euclidean norm of $y_i$'s, imply
$$q=\Exp\|G\|_B\geq \frac{\sqrt{q\,\Lip(\|\cdot\|_B)}}{2(\log n)^{1/4}}
\Exp\max_{i\leq m}|g_i|>\frac{\sqrt{q\,\Lip(\|\cdot\|_B)}(\log n)^{1/4}}{4}.$$
Hence, we get
$$\frac{q}{\Lip(\|\cdot\|_B)}>\frac{1}{16}\sqrt{\log n}.$$

\item Assume that $\|y_{m}\|_2< \frac{\sqrt{q\,\Lip(\|\cdot\|_B)}}{(\log n)^{1/4}}$.
Thus, the Lipschitz constant of $\|\cdot\|_{B\cap H_m}$
is less than $\frac{\sqrt{q\,\Lip(\|\cdot\|_B)}}{(\log n)^{1/4}}$, and Theorem~\ref{th: gauss concentration},
together with the relation $\Exp\|G\|_B\geq \Exp\|\Proj_{H_m}G\|_B$, imply
\begin{equation}\label{eq: aux 94}
\Prob\big\{\|\Proj_{H_m} G\|_{B\cap H_m}\geq q+4\sqrt{q\,\Lip(\|\cdot\|_B)}(\log n)^{1/4}\big\}\leq \frac{1}{n^2}.
\end{equation}
On the other hand, in view of the choice of $y_1$ and standard deviation estimates for a Gaussian variable, we have
$$\Prob\Big\{\|G\|_{B}\geq \frac{1}{2}\sqrt{\log n}\,\Lip(\|\cdot\|_B)\Big\}
\geq \frac{1}{2}\Prob\Big\{\frac{\|g y_1\|_{B}}{\|y_1\|_2}\geq \frac{1}{2}\sqrt{\log n}\,\Lip(\|\cdot\|_B)\Big\}
> \frac{1}{\sqrt{n}}.$$
Clearly, in view of \eqref{eq: grad max property} we have
$$\|G\|_B=\sum_{i=1}^n \|y_i\|_2^{-2}\langle \grad_B(G),y_i\rangle \langle G,y_i\rangle$$
and
\begin{align*}
\|\Proj_{H_m} G\|_{B\cap H_m}&=\sum_{i=m+1}^n\|y_i\|_2^{-2}\langle \grad_B( \Proj_{H_m} G),y_i\rangle \langle G,y_i\rangle\\
&\geq \sum_{i=m+1}^n\|y_i\|_2^{-2}\langle \grad_B( G),y_i\rangle \langle G,y_i\rangle.
\end{align*}
Thus, \eqref{eq: aux 94} yields
$$\Prob\Big\{\|G\|_B-\sum_{i=1}^m \|y_i\|_2^{-2}\langle \grad_B(G),y_i\rangle \langle G,y_i\rangle\geq q+4\sqrt{q\,\Lip(\|\cdot\|_B)}(\log n)^{1/4}\Big\}\leq \frac{1}{n^2}.$$
The last inequality, together with the above deviation estimates for $\|G\|_B$, implies
\begin{align*}
\Prob\bigg\{&\|G\|_B\sum_{i=1}^m \|y_i\|_2^{-2}\langle \grad_B(G),y_i\rangle \langle G,y_i\rangle
\geq \frac{1}{2}\sqrt{\log n}\,\Lip(\|\cdot\|_B)\cdot\\
&\Big(\frac{1}{2}\sqrt{\log n}\,\Lip(\|\cdot\|_B)-q-4\sqrt{q\,\Lip(\|\cdot\|_B)}(\log n)^{1/4}\Big)\bigg\}
\geq \frac{1}{\sqrt{n}}-\frac{1}{n^2}>\frac{1}{2\sqrt{n}},
\end{align*}
whence
\begin{align*}
\Exp\bigg(&\|G\|_B\sum_{i=1}^m \|y_i\|_2^{-2}\langle \grad_B(G),y_i\rangle \langle G,y_i\rangle\bigg)\\
&>\frac{\Lip(\|\cdot\|_B)\sqrt{\log n}}{4\sqrt{n}}\Big(\frac{1}{2}\sqrt{\log n}\,\Lip(\|\cdot\|_B)
-q-4\sqrt{q\,\Lip(\|\cdot\|_B)}(\log n)^{1/4}\Big).
\end{align*}
On the other hand, by Lemma~\ref{l: ell equations} and in view of the equivalence of moments of $\|G\|_B$
(see \cite[Corollary~3.2]{LT}),
$$\Exp\bigg(\|G\|_B\sum_{i=1}^m \|y_i\|_2^{-2}\langle \grad_B(G),y_i\rangle \langle G,y_i\rangle\bigg)
=\frac{m}{n}\Exp\|G\|^2_B\leq \frac{Cq^2}{\sqrt{n}}$$
for a universal constant $C>0$.
Thus,
$$Cq^2\geq \frac{\Lip(\|\cdot\|_B)\sqrt{\log n}}{4}\Big(\frac{1}{2}\sqrt{\log n}\,\Lip(\|\cdot\|_B)
-q-4\sqrt{q\,\Lip(\|\cdot\|_B)}(\log n)^{1/4}\Big).$$
Solving for $q$, we get
$$q\geq c'\Lip(\|\cdot\|_B)\sqrt{\log n}$$
for some constant $c'>0$.
\end{itemize}
\end{proof}

As a consequence of the above proposition, we get
\begin{lemma}\label{l: norm to gradient}
There are universal constants $n_0\in\N$ and $c>0$ with the following property.
Let $n\geq n_0$, let $B$ be an origin-symmetric convex body in $\R^n$ with a smooth boundary in the $\ell$-position.
Assume that
$$\Exp\|G\|_B\leq n^{c}\,\Exp\|\grad_B(G)\|_2.$$
Then for all $p\in[1,c\log n]$ we have
$$\Exp\|G\|_B^{2p}\leq n^{1/32}\,\Exp_{\gamma_n}\|\grad(\|\cdot\|_B^p)\|_2^2
= n^{1/32}p^2\,\Exp\big(\|G\|_B^{2p-2}\|\grad_B(G)\|_2^2\big),$$
where $\grad(\|\cdot\|_B^p)$ is the gradient of the $p$-th power of the norm $\|\cdot\|_B$.
\end{lemma}
\begin{proof}
We can assume that $n$ is large.
Proposition~\ref{p: ell lipschitz} and Theorem~\ref{th: gauss concentration}
imply that $\|G\|_B\geq\frac{1}{2}\Exp\|G\|_B$ with probability at least $1-n^{-c'}$
for a universal constant $c'\in(0,1/64]$. Now, set $c:=c'/2$ and assume that
$\Exp\|G\|_B\leq n^{c}\,\Exp\|\grad_B(G)\|_2$.
In view of Lemma~\ref{l: grad trivial upper bound} we have $\|\grad_B(x)\|_2\leq C\Exp\|G\|_B$
for all non-zero vectors $x$ and a universal constant $C>0$.
Hence, denoting by $\Event$ the event that $\|\grad_B(G)\|_2\geq \frac{1}{2}n^{-c}\Exp\|G\|_B$, we obtain
\begin{align*}
\Exp\|G\|_B
&\leq n^c\,\Exp\big(\|\grad_B(G)\|_2\,\chi_{\Event^c}\big)+n^c\,\Exp\big(\|\grad_B(G)\|_2\,\chi_{\Event}\big)\\
&\leq \frac{1}{2}\Exp\|G\|_B+Cn^{c}\,\Prob(\Event)\Exp\|G\|_B,
\end{align*}
implying $\Prob(\Event)\geq \frac{1}{2C}n^{-c}$.
Thus, with probability at least
$\frac{1}{2C}n^{-c}-n^{-c'}>n^{-c'}$ we have
$$\|G\|_B\geq\frac{1}{2}\Exp\|G\|_B\quad\mbox{and}\quad \|\grad_B(G)\|_2\geq \frac{1}{2}n^{-c}\Exp\|G\|_B,$$
whence for all $p\geq 1$ we get
$$\Exp\big(\|G\|_B^{2p-2}\|\grad_B(G)\|_2^2\big)\geq 2^{-2p}n^{-2c'}\big(\Exp\|G\|_B\big)^{2p}.$$
On the other hand, the concentration of $\|G\|_B$ (again, provided by Proposition~\ref{p: ell lipschitz}
and Theorem~\ref{th: gauss concentration}) implies that
$$\Exp\|G\|_B^{2p}<2^p (\Exp\|G\|_B)^{2p}$$
for all $p\leq c''\log n$ for a sufficiently small universal constant $c''>0$.
Thus, for all such $p$ we have
$$\Exp\|G\|_B^{2p}\leq 2^{3p}n^{2c'}\Exp\big(\|G\|_B^{2p-2}\|\grad_B(G)\|_2^2\big),$$
and the result follows.
\end{proof}

\section{The superconcentration of $\|\cdot\|_B^p$}

Let $B$ be an origin-symmetric convex body in $\R^n$ with a smooth boundary and let $p\geq 1$.
For any point $x\in\R^n\setminus\{0\}$ the $i$-th partial derivative of $\|\cdot\|_B^p$ at $x$ is equal to
$p\|x\|_B^{p-1}\langle \grad_B(x),e_i\rangle$.
Hence, applying Theorem~\ref{th: talagrand l1l2}, we get
$$\Var\big(\|G\|_B^p\big)\leq \sum\limits_{i=1}^n \frac{Cp^2\,\Exp(\|G\|_B^{2p-2}\langle \grad_B(G),e_i\rangle^2)}
{1+\log\big(\sqrt{\Exp(\|G\|_B^{2p-2}\langle \grad_B(G),e_i\rangle^2)}/\Exp|\|G\|_B^{p-1}\langle \grad_B(G),e_i\rangle|\big)},$$
where $C>0$ is a universal constant.
For each $i\leq n$, write
$$\langle \grad_B(x),e_i\rangle=F_i(B,x)+S_i(B,x),\quad\quad x\in\R^n\setminus\{0\},$$
where
\begin{align*}
F_i(B,x)&:=\langle \grad_B(x),e_i\rangle\chi_{\{|\langle \grad_B(x),e_i\rangle|\leq n^{-1/8}\Exp\|G\|_B\}};\\
S_i(B,x)&:=\langle \grad_B(x),e_i\rangle\chi_{\{|\langle \grad_B(x),e_i\rangle|> n^{-1/8}\Exp\|G\|_B\}}.
\end{align*}
Here, ``F'' stands for ``flat'' and ``S'' --- for ``spiky''.
Then the upper bound for the variance can be written as
\begin{align}
\Var\big(\|G\|_B^p\big)
&\leq C'p^2\, \Exp\Big(\|G\|_B^{2p-2}\sum_{i=1}^n F_i^2(B,G)\Big)\nonumber\\
&\hspace{-1cm}+\sum\limits_{i=1}^n \frac{C'p^2\,\Exp(\|G\|_B^{2p-2}S_i^2(B,G))}
{1+\log\big(\sqrt{\Exp(\|G\|_B^{2p-2}\langle \grad_B(G),e_i\rangle^2)}/\Exp|\|G\|_B^{p-1}\langle \grad_B(G),e_i\rangle|\big)}.
\label{eq: var bound}
\end{align}
We will treat the flat and the spiky parts separately.

\begin{lemma}[The flat part]\label{l: flat}
There are universal constants $n_0\in\N$ and $c>0$ with the following property.
Let $n\geq n_0$, let $B$ be a smooth origin-symmetric convex body in $\R^n$ in the $\ell$-position,
and assume that the standard basis in $\R^n$
is $1$-unconditional with respect to $\|\cdot\|_B$. Then for all $p$ in the interval $1\leq p\leq c\log n$
we have
$$p^2\,\Exp\Big(\|G\|_B^{2p-2}\sum\limits_{i=1}^n F_i^2(B,G)\Big)\leq n^{-1/16}\,\Exp\|G\|_B^{2p}.$$
\end{lemma}
\begin{proof}
We will assume that $n$ is large.
Note that
$$\sum\limits_{i=1}^n F_i^2(B,G)
\leq n^{-1/8}\Exp\|G\|_B\sum\limits_{i=1}^n |F_i(B,G)|,$$
whence for all $p\geq 1$ we get
\begin{align*}
\Exp\Big(&\|G\|_B^{2p-2}\sum\limits_{i=1}^n F_i^2(B,G)\Big)\\
&\leq n^{-1/8}\Exp\|G\|_B\,\Exp\Big(\|G\|_B^{2p-2}\sum\limits_{i=1}^n |F_i(B,G)|\Big)\\
&\leq n^{-1/8}\Exp\|G\|_B\,\big(\Exp\|G\|_B^{2p-1}\big)^{(2p-2)/(2p-1)}
\bigg(\Exp\Big(\sum\limits_{i=1}^n |F_i(B,G)|\Big)^{2p-1}\bigg)^{1/(2p-1)},
\end{align*}
where the last relation follows from H\"older's inequality.
Next, in view of Lemma~\ref{l: l1 grad},
$$\Exp\Big(\sum\limits_{i=1}^n |F_i(B,G)|\Big)^{2p-1}
\leq \Exp\Big(\sum\limits_{i=1}^n |\langle \grad_B(G),e_i\rangle|\Big)^{2p-1}
\leq (\pi/2)^{p-1/2}\Exp\|G\|_B^{2p-1},$$
whence
$$\Exp\Big(\|G\|_B^{2p-2}\sum\limits_{i=1}^n F_i^2(B,G)\Big)
\leq \sqrt{\pi/2}\,n^{-1/8}\Exp\|G\|_B\,\Exp\|G\|_B^{2p-1}.$$
It remains to note that we can choose the constant $c>0$ small enough and $n_0$ large enough to guarantee that
$$\sqrt{\pi/2}\,p^2n^{-1/8}\leq n^{-1/16}.$$
\end{proof}

\begin{lemma}[The spiky part]\label{l: spiky}
There are universal constants $n_0\in\N$ and $c'>0$ with the following property.
Let $n\geq n_0$, let $B$ be as in the last lemma, and let $i\leq n$, $p\geq 1$ and $\tau>0$ be such that
$$\Exp\big(\|G\|_B^{2p-2}F_i^2(B,G)\big)
\leq n^{-\tau}\Exp\big(\|G\|_B^{2p-2}\langle \grad_B(G),e_i\rangle^2\big).$$
Then
$$
\big(\Exp\big|\|G\|_B^{p-1}\langle \grad_B(G),e_i\rangle\big|\big)^2\leq
2(n^{-c'}+n^{-\tau})\Exp(\|G\|_B^{2p-2}\langle \grad_B(G),e_i\rangle^2).
$$
\end{lemma}
\begin{proof}
We will assume that $n$ is large.
In view of Lemma~\ref{l: ell equations}
as well as Lemma~\ref{l: unconditional grad} and the definition of $S_i(B,G)$, we have
$$\Exp\|G\|_B\,\Exp\big(\|G\|_B|\langle G,e_i\rangle|\chi_{\{S_i(B,G)\neq 0\}}n^{-1/8}\big)
\leq\Exp\big(\|G\|_B\langle \grad_B(G),e_i\rangle\langle G,e_i\rangle\big)
= \frac{1}{n}\Exp\|G\|_B^2.$$
On the other hand, 
conditioned on any realization of $\langle G,e_j\rangle$ ($j\neq i$),
$\|G\|_B$ is a monotone function of $|\langle G,e_i\rangle|$,
and, in view of Lemma~\ref{l: grad is monotone} and the definition of $S_i$, $\chi_{\{S_i(B,G)\neq 0\}}$
is a monotone function of $|\langle G,e_i\rangle|$. Hence, setting $g_j:=\langle G,e_j\rangle$ ($j\leq n$),
we obtain
\begin{align*}
\Exp\big(\|G\|_B|\langle G,e_i\rangle|\chi_{\{S_i(B,G)\neq 0\}}\big)
&=\Exp_{\{g_j,\;j\neq i\}}\Exp_{g_i}\big(\|G\|_B|\langle G,e_i\rangle|\chi_{\{S_i(B,G)\neq 0\}}\big)\\
&\geq \Exp\big(\|G\|_B\,\chi_{\{S_i(B,G)\neq 0\}}\big)\,\Exp|\langle G,e_i\rangle|\\
&=\sqrt{\frac{2}{\pi}}\,\Exp\big(\|G\|_B\,\chi_{\{S_i(B,G)\neq 0\}}\big).
\end{align*}
Further,
\begin{align*}
\Exp\big(\|G\|_B\,\chi_{\{S_i(B,G)\neq 0\}}\big)
&\geq\Exp\big(\|G\|_B\,\chi_{\{S_i(B,G)\neq 0\mbox{ and }2\|G\|_B\geq \Exp\|G\|_B\}}\big)\\
&\geq \frac{1}{2}\Exp\|G\|_B\,\Prob\big\{S_i(B,G)\neq 0\mbox{ and }2\|G\|_B\geq \Exp\|G\|_B\big\}\\
&\geq \frac{1}{2}\Exp\|G\|_B\,\big(\Prob\big\{S_i(B,G)\neq 0\big\}-\Prob\big\{2\|G\|_B< \Exp\|G\|_B\big\}\big)\\
&\geq \frac{1}{2}\Exp\|G\|_B\,\big(\Prob\big\{S_i(B,G)\neq 0\big\}-n^{-c'}\big),
\end{align*}
where the last inequality follows from Proposition~\ref{p: ell lipschitz} and Theorem~\ref{th: gauss concentration}.
Combining all the above inequalities, we obtain
$$\frac{1}{n}=\frac{1}{n}\Exp\|G\|_B^2\geq \frac{1}{\sqrt{2\pi}}\,n^{-1/8}\big(\Exp\|G\|_B\big)^2
\,\big(\Prob\big\{S_i(B,G)\neq 0\big\}-n^{-c'}\big).$$
Taking into consideration the equivalence of moments of $\|G\|_B$, we get that
$$\Prob\big\{S_i(B,G)\neq 0\big\}\leq n^{-c''}$$
for a universal constant $c''>0$.
Thus,
$$\big(\Exp|\|G\|_B^{p-1} S_i(B,G)|\big)^2\leq n^{-c''}\Exp\big(\|G\|_B^{2p-2} S_i^2(B,G)\big).$$
Finally, we have, by the assumption of the lemma and the above relation,
\begin{align*}
\Exp\big(\|G\|_B^{2p-2}\langle \grad_B(G),e_i\rangle^2\big)
&\geq n^{c''}\big(\Exp|\|G\|_B^{p-1} S_i(B,G)|\big)^2;\\
\Exp\big(\|G\|_B^{2p-2}\langle \grad_B(G),e_i\rangle^2\big)
&\geq n^{\tau}\big(\Exp|\|G\|_B^{p-1} F_i(B,G)|\big)^2,
\end{align*}
whence
$$\big(\Exp|\|G\|_B^{p-1}\langle \grad_B(G),e_i\rangle|\big)^2
\leq 2(n^{-c''}+n^{-\tau})\Exp\big(\|G\|_B^{2p-2}\langle \grad_B(G),e_i\rangle^2\big).$$
\end{proof}

\begin{proof}[Proof of Theorem~\ref{th: superconcentration}]
We suppose that $n$ is large.
Moreover, in view of Lemma~\ref{l: uncond} and rotation-invariance of the Gaussian distribution, we can assume without loss of generality
that the standard basis in $\R^n$ is $1$-unconditional with respect to $\|\cdot\|_B$.
Let $c>0$ be minimum of the constants from Lemmas~\ref{l: norm to gradient} and~\ref{l: flat},
and let $p\in[1,c\log n]$.
We assume that
\begin{equation}\label{eq: aux 6t}
\Exp\|G\|_B\leq n^{c}\,\Exp\|\grad_B(G)\|_2.
\end{equation}
Let us start by applying Lemma~\ref{l: spiky} with $\tau:=1/64$.
Note that for those $i\leq n$ with
$$\Exp\big(\|G\|_B^{2p-2}F_i^2(B,G)\big)
\leq n^{-1/64}\Exp\big(\|G\|_B^{2p-2}\langle \grad_B(G),e_i\rangle^2\big)$$
we have
$$\log\big(\sqrt{\Exp(\|G\|_B^{2p-2}\langle \grad_B(G),e_i\rangle^2)}/\Exp|\|G\|_B^{p-1}\langle \grad_B(G),e_i\rangle|\big)
\geq \widetilde c\log n$$
for a universal constant $\widetilde c>0$.
Hence,
\begin{align*}
&\sum\limits_{i=1}^n \frac{C'p^2\,\Exp(\|G\|_B^{2p-2}S_i^2(B,G))}
{1+\log\big(\sqrt{\Exp(\|G\|_B^{2p-2}\langle \grad_B(G),e_i\rangle^2)}/\Exp|\|G\|_B^{p-1}\langle \grad_B(G),e_i\rangle|\big)}\\
&\hspace{1cm}\leq
C'p^2n^{1/64}\sum\limits_{i=1}^n \Exp(\|G\|_B^{2p-2}F_i^2(B,G))
+\frac{C''p^2}{\log n}\sum_{i=1}^n\Exp(\|G\|_B^{2p-2}S_i^2(B,G)).
\end{align*}
Together with relation~\eqref{eq: var bound} this gives
$$
\Var\big(\|G\|_B^p\big)
\leq \widetilde C p^2n^{1/64}\sum\limits_{i=1}^n \Exp(\|G\|_B^{2p-2}F_i^2(B,G))
+\frac{\widetilde C p^2}{\log n}\sum_{i=1}^n \Exp(\|G\|_B^{2p-2}S_i^2(B,G)).$$
Next, applying Lemma~\ref{l: flat} we obtain
$$
\Var\big(\|G\|_B^p\big)
\leq \widetilde C n^{-3/64}\,\Exp\|G\|_B^{2p}
+\frac{\widetilde C p^2}{\log n}\sum_{i=1}^n\Exp(\|G\|_B^{2p-2}S_i^2(B,G)).$$
Finally, in view of Lemma~\ref{l: norm to gradient} and~\eqref{eq: aux 6t}, this gives
$$
\Var\big(\|G\|_B^p\big)
\leq \Big(\widetilde C p^2 n^{-1/64}+
\frac{\widetilde C p^2}{\log n}\Big)\Exp(\|G\|_B^{2p-2}\|\grad_B(G)\|_B^2).$$
It remains to note that $p\|G\|_B^{p-1}\| \grad_B(G)\|_2$ equals the Euclidean norm of the gradient
of $\|\cdot\|^p_B$ at $G$, and apply the definition of superconcentration.
\end{proof}

\section{The randomized Dvoretzky theorem}

We will show how the variance bound from the previous section is translated into
a small deviations inequality for $\|G\|_B$. At a high level, the procedure is rather standard;
for example, let us refer to \cite[Chapter~3]{Ledoux 2001} for a very general scheme
that allows to deduce exponential concentration from {\it the Poincar\'e inequality}.
On the other hand, as we have better bounds on the variance of $\|G\|_B^p$
than those provided by the Poincar\'e inequality, our deviation estimates are stronger.
Let us remark that the use of superconcentration in our analysis was inspired by a recent paper
of G.~Paouris, P.~Valettas and J.~Zinn \cite{PVZ} dealing with almost Euclidean sections of $\ell_p^n$-balls,
which complemented earlier results of A.~Naor \cite{Naor}.

\medskip

Let us start with a simple lemma that follows immediately from Proposition~\ref{p: ell lipschitz}
and Theorem~\ref{th: gauss concentration}:
\begin{lemma}\label{l: aux 32}
There are universal constants $n_0\in\N$ and $C\geq e$ with the following property.
Let $n\geq n_0$, let $B$ be an origin-symmetric convex body in $\R^n$ in the $\ell$-position, with a smooth boundary, and let $q\leq\log n$.
Then $\Exp\|G\|_B^{2q}\leq C^q\,\big(\Exp\|G\|_B^{q}\big)^2$.
\end{lemma}

\begin{lemma}\label{l: aux 09}
There are constants $n_0\in\N$, $c'>0$ and $C'>0$ with the following property.
Let $B$ be an origin-symmetric convex body in $\R^n$ in the $\ell$-position, with a smooth boundary,
and such that $(\R^n,\|\cdot\|_B)$ admits a $1$-unconditional basis.
Then for all $p\in[1,c'\log n]$ we have
$$\Var(\|G\|_B^p)\leq\frac{C'\,p^2}{(\log n)^2}\,\Exp\|G\|_B^{2p}.$$
\end{lemma}
\begin{proof}
We can suppose that $n$ is large. Let $c>0$ be the constant from Theorem~\ref{th: superconcentration}
(we can safely assume that $c\leq 1/2$), and set $c':=c/(8\log C)$,
where $C$ is taken from Lemma~\ref{l: aux 32}.
Further, let $p\in [1,c'\log n]$ and consider two cases:

\begin{itemize}

\item Suppose that $\Exp\|G\|_B> n^c\,\Exp\|\grad_B(G)\|_2$.
Denote by $\Event$ the event $\|\grad_B(G)\|_2\geq n^{-c/2}\Exp\|G\|_B$.
Clearly, $\Prob(\Event)\leq n^{-c/2}$ by Markov's inequality.
Now, the Poincar\'e inequality for $\|\cdot\|_B^p$ implies
\begin{align*}
\Var(\|G\|_B^p)&\leq p^2\,\Exp\big(\|G\|_B^{2p-2}\|\grad_B(G)\|_2^2\big)\\
&\leq p^2\,\Exp\big(\|G\|_B^{2p-2}\|\grad_B(G)\|_2^2\,\chi_{\Event}\big)
+p^2n^{-c}\,\big(\Exp\|G\|_B\big)^2\,\Exp\|G\|_B^{2p-2}\\
&\leq \widetilde C p^2\,\big(\Exp\|G\|_B\big)^2\,\Exp\big(\|G\|_B^{2p-2}\,\chi_{\Event}\big)
+p^2n^{-c}\,\Exp\|G\|_B^{2p},
\end{align*}
where at the last step we applied Lemma~\ref{l: grad trivial upper bound}.
Further, by the Cauchy--Schwarz inequality and Lemma~\ref{l: aux 32} we get
$$\Exp\big(\|G\|_B^{2p-2}\,\chi_{\Event}\big)\leq \sqrt{\Prob(\Event)}\,\big(\Exp\|G\|_B^{4p-4}\big)^{1/2}
\leq n^{-c/4}\,C^{p-1}\Exp\|G\|_B^{2p-2}.$$
Finally,
$$\Var(\|G\|_B^p)\leq 2\widetilde C p^2n^{-c/4}\,C^{p-1}\Exp\|G\|_B^{2p}
\leq 2\widetilde C p^2n^{-c/8}\,\Exp\|G\|_B^{2p}
\leq \bar C n^{-c/16}\Exp\|G\|_B^{2p},$$
where the second inequality follows from the choice of $c'$.

\item If $\Exp\|G\|_B\leq n^c\,\Exp\|\grad_B(G)\|_2$ then, by Theorem~\ref{th: superconcentration}, we have
$$\Var(\|G\|_B^p)\leq\frac{C'p^2}{\log n}\,\Exp\big(\|G\|_B^{2p-2}\|\grad_B(G)\|_2^2\big)$$
for a universal constant $C'>0$.
Note that, in view of Proposition~\ref{p: ell lipschitz},
$$\|\grad_B(G)\|_2\leq \Lip(\|\cdot\|_B)\leq \frac{C''\Exp\|G\|_B}{\sqrt{\log n}}$$
for a universal constant $C''>0$, whence
$$\Var(\|G\|_B^p)\leq\frac{\bar C\,p^2}{(\log n)^2}\,\Exp\|G\|_B^{2p}.$$
\end{itemize}
Thus in any case we obtain the required bound.
\end{proof}

Now, we have
\begin{theor}\label{th: small deviations}
There are universal constants $n_0\in\N$ and $c>0$ with the following property.
Let $n\geq n_0$, let $B$ be an origin-symmetric convex body in $\R^n$ in the $\ell$-position,
and assume that $(\R^n,\|\cdot\|_B)$ 
admits a $1$-unconditional basis.
Then for any $\varepsilon\in(0,1/2]$ we have
$$\Prob\big\{\big|\|G\|_B-\Med\|G\|_B\big|\geq \varepsilon \Med\|G\|_B\big\}\leq 2n^{-c\varepsilon}.$$
\end{theor}
\begin{proof}
As before, we assume that $n$ is large.
Note that, in view of Corollary~\ref{c:smooth}, our convex body $B$ can be approximated with
arbitrary precision by a smooth convex body in the $\ell$-position. Thus, without loss of generality
we can assume that $B$ itself is smooth.
Let $c'$ and $C'$ be the constants from Lemma~\ref{l: aux 09}, and set
$\widetilde c:=\min(c',\frac{1}{2\sqrt{C'}})$ and 
$p:=\widetilde c\log n$.
Note that Lemma~\ref{l: aux 09} then implies
$$\Var(\|G\|_B^p)=\Exp\|G\|_B^{2p}-\big(\Exp\|G\|_B^p\big)^2\leq\frac{1}{4}\Exp\|G\|_B^{2p},$$
whence $\Exp\|G\|_B^{2p}\leq \frac{4}{3}\big(\Exp\|G\|_B^p\big)^2$.
On the other hand,
$$\Var(\|G\|_B^p)\geq \frac{1}{2}\big(\Med\|G\|_B^p-\Exp\|G\|_B^p\big)^2,$$
which, together with the above inequality, gives
$$\big|\Med\|G\|_B^p-\Exp\|G\|_B^p\big|\leq \sqrt{\frac{2}{3}}\,\Exp\|G\|_B^p.$$
Now, for any $\varepsilon>0$ we get
\begin{align*}
\Prob\big\{\|G\|_B\geq (1+\varepsilon) \Med\|G\|_B\big\}
&= \Prob\big\{\|G\|_B^p\geq(1+\varepsilon)^p\,\Med\|G\|_B^p\big\}\\
&\leq \Prob\big\{\|G\|_B^p\geq 0.18(1+\varepsilon)^p\,\Exp\|G\|_B^p\big\}.
\end{align*}
Assume that $0.09(1+\varepsilon)^p\geq 1$.
Then, by Chebyshev's inequality and the above, we get
\begin{align*}
\Prob\big\{\|G\|_B\geq (1+\varepsilon) \Med\|G\|_B\big\}
&\leq\frac{\Var(\|G\|_B^p)}{0.09^2(1+\varepsilon)^{2p}\big(\Exp\|G\|_B^p\big)^2}\\
&<\frac{200\Var(\|G\|_B^p)}{(1+\varepsilon)^{2p}\Exp\|G\|_B^{2p}}\\
&< \frac{100}{(1+\varepsilon)^{2p}}.
\end{align*}
To get lower deviation estimates, we apply a theorem of D.~Cordero-Erausquin, M.~Frade\-lizi and B.~Maurey \cite{CEFM}.
According to the theorem, the function
$$t\to \Prob\big\{\|G\|_B\leq e^t \Med\|G\|_B\big\}$$
is log-concave on the real line.
Hence, assuming $\varepsilon_0>0$ is the number satisfying
$\Prob\big\{\|G\|_B\leq (1+\varepsilon_0) \Med\|G\|_B\big\}=3/4$, we get
\begin{align*}
\log\Prob\big\{\|G\|_B\leq (1-\varepsilon) \Med\|G\|_B\big\}
&\leq \log\Prob\big\{\|G\|_B\leq e^{-\varepsilon}\,\Med\|G\|_B\big\}\\
&\leq -\frac{\varepsilon}{\varepsilon_0} \Big(\log\Prob\big\{\|G\|_B\leq e^{\varepsilon_0} \Med\|G\|_B\big\}-\log\frac{1}{2}\Big)\\
&\leq -\frac{\varepsilon}{\varepsilon_0} \Big(\log\frac{3}{4}-\log\frac{1}{2}\Big),
\end{align*}
whence
$$\Prob\big\{\|G\|_B\leq (1-\varepsilon) \Med\|G\|_B\big\}\leq \exp\Big(-\frac{\varepsilon}{\varepsilon_0}\log\frac{3}{2}\Big)$$
for all $\varepsilon>0$. Note that the upper deviation estimates obtained in the first part of the proof
imply that $\varepsilon_0\leq \widetilde C(\log n)^{-1}$ for a universal constant $\widetilde C>0$.
Finally,
$$\Prob\big\{\big|\|G\|_B-\Med\|G\|_B\big|\geq \varepsilon \Med\|G\|_B\big\}
\leq \frac{100}{(1+\varepsilon)^{2\widetilde c\log n}}
+\exp\Big(-\widetilde C^{-1}\log\frac{3}{2}\;\varepsilon\log n\Big)$$
for all $\varepsilon\geq \frac{C''}{\log n}$. The result follows.
\end{proof}

The main result of this note --- Theorem~\ref{th: random ell} --- is obtained from Theorem~\ref{th: small deviations}
via a simple covering argument. We prefer to omit this (completely standard by now) part of the proof;
we refer, in particular, to \cite{MS 1986, Schechtman 2013, AGM 2015} for information on this matter.

\section{Proof of Proposition~\ref{p: John poor concentration}}\label{s:john}

In this section, we prove Proposition~\ref{p: John poor concentration} from the introduction,
by providing an example of a convex set in $\R^n$ in John's position
%and symmetric with respect to coordinate hyperplanes
which shows that in general one cannot expect a better than quadratic dependence
on $\varepsilon$ in Theorem~\ref{th: random John}. In this connection
it is natural to recall an example by T.~Figiel
which highlights the limitations of the {\it existential} Dvoretzky theorem
(we refer to \cite[Lecture~3]{Schechtman 2013} for more details).
However, our example operates in a different regime as we bound the dimension of almost Euclidean
sections by $\varepsilon^2 \log n$ whereas Figiel's convex set admits $\varepsilon^2 n$--dimensional sections.

Let us start by stating two facts.
The first of the two lemmas below can be verified by combining a concentration inequality for the $\ell_\infty^r$-norm
of the Gaussian vector with a standard covering argument (see \cite{Schechtman 2007}),
whereas the second one is a simple (and rather crude)
corollary of Theorem~\ref{th: gauss concentration} (again, combined with a covering procedure).
\begin{lemma}\label{l: l infty almost euclid}
There are universal constants $c_{\ref{l: l infty almost euclid}}>0$ and $r_0\in\N$ with the following property:
Let $r\geq r_0$, $\varepsilon\in(0,1/2]$ and $1\leq k\leq c_{\ref{l: l infty almost euclid}}\varepsilon\log r/\log\frac{1}{\varepsilon}$.
Let $X_1,X_2,\dots,X_k$ be i.i.d.\ standard Gaussian vectors in $\R^r$, and set
$M_r:=\Med \|X_1\|_\infty$.
Then
$$\Prob\Big\{(1-\varepsilon)M_r\leq \Big\|\sum\nolimits_{i=1}^k \alpha_i X_i\Big\|_\infty\leq
(1+\varepsilon)M_r\;\;\mbox{ for all }(\alpha_1,\dots,\alpha_k)\in S^{k-1}\Big\}\geq\frac{7}{8}.
$$
\end{lemma}

\begin{lemma}\label{l: gaussian almost spherical}
There are universal constants $n_0>0$ and $c_{\ref{l: gaussian almost spherical}}>0$
such that for any $n\geq n_0$ and $k\leq n^{c_{\ref{l: gaussian almost spherical}}}$
the following holds:
Let $X_1,X_2,\dots,X_k$ be i.i.d.\ standard Gaussian vectors in $\R^n$. Then
$$\Prob\Big\{(1-n^{-c_{\ref{l: gaussian almost spherical}}})\sqrt{n}
\leq\Big\|\sum\nolimits_{i=1}^k \alpha_i X_i\Big\|_2\leq (1+n^{-c_{\ref{l: gaussian almost spherical}}})\sqrt{n}
\;\;\mbox{ for all }(\alpha_1,\dots,\alpha_k)\in S^{k-1}\Big\}
\geq\frac{7}{8}.$$
\end{lemma}

\medskip

The following lemma is a trivial planimetric observation; we provide the proof for reader's convenience.

\begin{lemma}\label{l: intersection of ellipses}
Let $0<a<1<b$, and let a figure $F$ in the plane be given by
$$F:=B_2^2\cap \Big\{(x_1,x_2)\in\R^2:\,\frac{x_1^2}{a^2}+\frac{x_2^2}{b^2}\leq 1\Big\},$$
where $B_2^2$ is the unit Euclidean ball in the plane.
Then the Banach--Mazur distance from $F$ to $B_2^2$ can be estimated as
$$\dist(F,B_2^2)\geq \sqrt{\frac{2b^2-a^2b^2-1}{b^2-a^2}}=\sqrt{1+\frac{(b^2-1)(1-a^2)}{b^2-a^2}}.$$
\end{lemma}
\begin{proof}
Clearly, the convex figure is isometric to the intersection of the disk $B_2^2$
and ellipse $\Ell:=\{(x_1,x_2):\,a^2 x_1^2+b^2 x_2^2\leq 1\}$.
Further, the four points of intersection of the boundaries $\partial B_2^2$ and $\partial \Ell$
have coordinates
$$\bigg(\pm\sqrt{\frac{b^2-1}{b^2-a^2}},\pm\sqrt{\frac{1-a^2}{b^2-a^2}}\;\bigg).$$
Now, assume that $T$ is a linear transformation of $\R^2$ such that
$$\tfrac{1}{\dist(B_2^2\cap\Ell,B_2^2)}B_2^2\subset T(B_2^2\cap \Ell)\subset B_2^2.$$
In view of the symmetries in $B_2^2\cap\Ell$,
we can assume that $T$ is diagonal, with $Te_1=:\kappa e_1$ and $Te_2=:\beta e_2$
for some $\kappa<1<\beta$.
Since $T(\partial B_2^2\cap\partial\Ell)\subset B_2^2$, we have the inequality
$$\kappa^2\,\frac{b^2-1}{b^2-a^2}+\beta^2\,\frac{1-a^2}{b^2-a^2}\leq 1.$$
On the other hand, the distance $\dist(B_2^2\cap\Ell,B_2^2)$ is bounded from below by
$\max(\frac{1}{\kappa},\frac{b}{\beta})$, whence, for any $\lambda\in[0,1]$, we have
$$\dist(B_2^2\cap\Ell,B_2^2)^{-2}\leq \lambda \kappa^2+\frac{1-\lambda}{b^2}\beta^2.$$
Choose $\lambda:=s\frac{b^2-1}{b^2-a^2}$, where $s:=\frac{b^2-a^2}{2b^2-a^2b^2-1}$.
It is easy to check that in this case $1-\lambda=sb^2\frac{1-a^2}{b^2-a^2}$, so that
$$\dist(B_2^2\cap\Ell,B_2^2)^{-2}\leq s\bigg(\kappa^2\,\frac{b^2-1}{b^2-a^2}+\beta^2\,\frac{1-a^2}{b^2-a^2}\bigg)\leq s.$$
Thus,
$$\dist(B_2^2\cap\Ell,B_2^2)^2\geq \frac{2b^2-a^2b^2-1}{b^2-a^2}=1+\frac{(b^2-1)(1-a^2)}{b^2-a^2}.$$
\end{proof}

In the following statement, we estimate the extreme singular values of a standard rectangular Gaussian matrix.
The lemma is by no means new; however, we prefer to give an elementary proof based only on the
standard concentration inequalities and not involving any spectral theory.

\begin{lemma}\label{l: gaussian extreme}
There are universal constants $c_{\ref{l: gaussian extreme}},C_{\ref{l: gaussian extreme}}>0$
with the following property:
Let $C_{\ref{l: gaussian extreme}}\leq k\leq m$, and let $A$ be the $m\times k$ standard Gaussian
matrix. Then $\Prob\big\{s_{\max}(A)\leq \sqrt{m}+c_{\ref{l: gaussian extreme}}\sqrt{k}\big\}\leq \frac{1}{16}$
and $\Prob\big\{s_{\min}(A)\geq \sqrt{m}-c_{\ref{l: gaussian extreme}}\sqrt{k}\big\}\leq \frac{1}{16}$.
\end{lemma}
\begin{proof}
Let us prove only the first assertion of the lemma (the argument for $s_{\min}(A)$
is very similar). We assume that $k\leq m$, and that $k$ is sufficiently large.
Let $Y_1,Y_2,\dots,Y_k$ be i.i.d.\ standard Gaussian vectors in $\R^m$.
We set $\rho_1:=1$ and define random signs $\rho_2,\rho_3,\dots,\rho_k$ inductively as follows:
$$\rho_i:=\sign\big\langle Y_i,\sum_{j=1}^{i-1} \rho_j Y_j\big\rangle,\;\;\;i=2,3,\dots,k.$$
We shall estimate the norm of the linear combination $\sum_{i=1}^k \rho_i Y_i$.
Clearly, for any $u\leq k$ we have
\begin{align*}
\Big\|\sum_{i=1}^u \rho_i Y_i\Big\|_2^2&=\sum_{i=1}^u \|Y_i\|_2^2
+2\sum_{i=2}^u\big\langle \rho_i Y_i,\sum_{j=1}^{i-1} \rho_j Y_j\big\rangle\\
&=\sum_{i=1}^u \|Y_i\|_2^2
+2\sum_{i=2}^u\Big|\big\langle Y_i,\sum_{j=1}^{i-1} \rho_j Y_j\big\rangle\Big|\\
&=\sum_{i=1}^u \|Y_i\|_2^2
+2\sum_{i=2}^u\Big\| \sum_{j=1}^{i-1} \rho_j Y_j\Big\|_2 \,|g_i|,
\end{align*}
where $g_i:=\big\|\sum_{j=1}^{i-1}\rho_j Y_j\big\|_2^{-1}\big\langle Y_i,\sum_{j=1}^{i-1} \rho_j Y_j\big\rangle$
($i=2,3,\dots,k$) are standard Gaussian variables.
A rough estimate gives (provided that $k$ is sufficiently large):
$$\Prob\Big\{\Big\|\sum_{i=1}^u \rho_i Y_i\Big\|_2\geq \frac{\sqrt{mk}}{4}\;\;\mbox{ for all }\;\;u\geq \frac{k}{2}\Big\}\geq\frac{63}{64}.$$
Next, the principal observation is that, conditioned on any realization of $g_2,\dots,g_{i-1}$, the variable $g_i$
is distributed according to the normal law, whence $g_i$'s are jointly independent.
It follows that, provided that $k$ is sufficiently large,
$$\sum_{i=\lceil k/2\rceil+1}^k |g_i|\geq \frac{k}{8}$$
with probability at least $63/64$.
Together with the above estimates, this gives
$$\Big\|\sum_{i=1}^k \rho_i Y_i\Big\|_2^2\geq 
\sum_{i=1}^k \|Y_i\|_2^2+\frac{k\sqrt{mk}}{16}$$
with probability at least $31/32$.
Combined with Theorem~\ref{th: gauss concentration} applied to $\big(\sum_{i=1}^k \|Y_i\|_2^2\big)^{1/2}$
viewed as $1$-Lipschitz function of $mk$ i.i.d.\ standard Gaussian variables,
this yields
$$\Big\|\sum_{i=1}^k \rho_i Y_i\Big\|_2^2\geq 
mk+\frac{k\sqrt{mk}}{17}> k\bigg(\sqrt{m}+\frac{\sqrt{k}}{64}\bigg)^2$$
with probability at least $15/16$. It remain to note that, for the $m\times k$ Gaussian matrix $A$
with columns $Y_1,Y_2,\dots,Y_k$, we have $\sqrt{k}\,s_{\max}(A)\geq \big\|\sum_{i=1}^k \rho_i Y_i\big\|_2$
deterministically.
\end{proof}
\begin{rem}
Note that in the last lemma we estimate $s_{\max}$ {\it from below} and $s_{\min}$ from above.
A lower bound on $s_{\max}$ and upper bound on $s_{\min}$ of a random matrix with i.i.d.\ centered entries
can be derived as a corollary of the Marchenko--Pastur theorem for the limiting spectral distribution \cite{MP 1967}.
However, the Marchenko--Pastur theorem requires that the ratio $k/m$ converges to a fixed number,
and its applicability in the case $k=o(m)$ is unclear.
On the other hand, the above proof of the lemma is based on an elementary argument
which gives rather crude bounds but remains valid under very mild assumptions on
$k$, $m$.
\end{rem}

\begin{proof}[Proof of Proposition~\ref{p: John poor concentration}]
To simplify working with constants, we adopt the following convention in this proof:
by writing ``$a\ll b$'' we mean that $a\leq cb$ for some universal constant $c>0$ which can be made arbitrarily
small at expense of changing the constants in the final statement.

Let us fix a (large) $n$, and define $m:=\lfloor \Med\max_{i\leq n}g_i^2\rfloor$,
where $g_1,\dots,g_n$ are i.i.d.\ standard Gaussians.
Note that $m=O(\log n)$ (see, for example, \cite{Schechtman 2007} or \cite[p.~302]{DN 2003}).
Further, define two cylinders
\begin{align*}
B'&:=\big\{(x_1,x_2,\dots,x_n)\in\R^n:\,\max_{i\leq n-m}|x_i|\leq 1\big\}\\
B''&:=\Big\{(x_1,x_2,\dots,x_n)\in\R^n:\,\sum\nolimits_{i=n-m+1}^n x_i^2\leq 1\Big\},
\end{align*}
and set $B:=B'\cap B''$. It is easy to see that
$B_2^2\subset B$, and that $\pm e_i$ ($i=1,2,\dots,n$) are contact points of $\partial B$ and $\partial B_2^2$,
whence, by John's theorem \cite{John,Ball}, $B_2^2$ is the maximal volume ellipsoid inside $B$.
Let $\varepsilon\in\big((\log n)^{-1/2},c\big]$, $k:=\lceil c^{-1}\varepsilon^2\log n\rceil$ (where $c>0$ is a sufficiently
small universal constant whose value can be recovered from the proof), and set
$E:=\spn\{X_1,X_2,\dots,X_k\}$ where $X_1,X_2,\dots,X_k$ are i.i.d.\ standard Gaussian vectors in $\R^n$.
We will show that with probability at least $1/2$, the random section $B\cap E$ is not $(1+\varepsilon)$--Euclidean.

Set $r:=n-m$. Note that, by the definition of $k$, we have
$k\ll \varepsilon\log r/\log \frac{1}{\varepsilon}$. Then, applying Lemma~\ref{l: l infty almost euclid}
with $\varepsilon/2$ in place of $\varepsilon$,
we obtain
$$\Prob\Big\{(1-\varepsilon/2)M_r\leq \Big\|\sum\nolimits_{i=1}^k \alpha_i X_i\Big\|_{B'}\leq
(1+\varepsilon/2)M_r\;\;\mbox{ for all }\;\;(\alpha_1,\dots,\alpha_k)\in S^{k-1}\Big\}\geq\frac{7}{8},$$
where $M_r=\Med\max_{i\leq r}|\langle X_1,e_i\rangle|=\Med \|X_1\|_{B'}$.
Next, observe that, in view of the choice of $m$ and asymptotic estimates of the median of the $\|\cdot\|_{\infty}$--norm
of Gaussian vectors (see, in particular, \cite[p.~302]{DN 2003}), we have
$\big|\sqrt{m}/M_r-1\big|\ll (\log n)^{-1/2}$. Together with Lemma~\ref{l: gaussian almost spherical},
the above relation and the assumption that $\varepsilon$ is small, gives
\begin{align*}
\Prob\Big\{&(1-\varepsilon)\sqrt{n/m}\leq \Big\|\sum\nolimits_{i=1}^k \alpha_i X_i\Big\|_{B'}^{-1}
\,\cdot \Big\|\sum\nolimits_{i=1}^k \alpha_i X_i\Big\|_{2}\leq
(1+\varepsilon)\sqrt{n/m}\\
&\mbox{for all }\;\;(\alpha_1,\dots,\alpha_k)\in S^{k-1}\Big\}\geq\frac{3}{4},
\end{align*}
or, geometrically,
\begin{equation}\label{eq: aux5987}
\Prob\big\{(1-\varepsilon)\sqrt{n/m}\,B_2^n\subset B'\cap E
\subset (1+\varepsilon)\sqrt{n/m}\,B_2^n\big\}\geq\frac{3}{4}.
\end{equation}

Further, the intersection $B''\cap E$ is clearly an ellipsoid.
Let $Y_1,Y_2,\dots,Y_k$ be the orthogonal projections of $X_i$'s onto the linear span of $\{e_{n-m+1},\dots,e_n\}$,
and let $A$ be the $m\times k$ random Gaussian matrix with columns $Y_1,\dots,Y_k$.
Note that
$$\sup\Big\{\Big\|\sum\nolimits_{i=1}^k \alpha_i Y_i\Big\|_2^{-1}
\,\cdot \Big\|\sum\nolimits_{i=1}^k \alpha_i X_i\Big\|_{2}:\,(\alpha_1,\dots,\alpha_k)\in S^{k-1}\Big\}$$
is the length of the largest semi-axis of $B''\cap E$ (let us denote the corresponding random vector by $W$,
i.e.\ $W\in \partial (B''\cap E)$ is the largest vector in $\R^n$ with the end-point on the boundary of $B''\cap E$).
Similarly, we let $Z$ be the smallest semi-axis of $B''\cap E$,
i.e.\ the shortest vector in $\R^n$ with the end-point on the boundary of $B''\cap E$.
In view of Lemma~\ref{l: gaussian almost spherical}, we have
$$\Prob\big\{\|W\|_2\geq (1-n^{-c'})s_{\min}(A)^{-1}\sqrt{n}
\;\;\mbox{and}\;\;\|Z\|_2\leq (1+n^{-c'})s_{\max}(A)^{-1}\sqrt{n}\big\}\geq \frac{7}{8}$$
for some universal constant $c'>0$.
Hence, by Lemma~\ref{l: gaussian extreme} we get
\begin{equation}\label{eq: aux404}
\Prob\Big\{\|W\|_2\geq \frac{\sqrt{n}}{\sqrt{m}
-c''\sqrt{k}}\;\;\mbox{ and }\;\;
\|Z\|_2\leq \frac{\sqrt{n}}{\sqrt{m}+c''\sqrt{k}}
\Big\}\geq \frac{3}{4},
\end{equation}
where $c''>0$ is a universal constant.
Denote by $\widetilde E$ the (random) $2$-dimensional span of $W$ and $Z$.
In view of \eqref{eq: aux5987}, with probability at least $3/4$ the random figure
$B\cap \widetilde E=(B'\cap \widetilde E)\cap (B''\cap \widetilde E)$ is at the distance at most $\frac{1+\varepsilon}{1-\varepsilon}$ from
$(\sqrt{n/m}B_2^n)\cap B''\cap \widetilde E$. On the other hand, applying Lemma~\ref{l: intersection of ellipses},
we get
$$\dist\big((\sqrt{n/m}B_2^n)\cap B''\cap \widetilde E,B_2^2\big)^2\geq
1+\frac{(\|W\|_2^2-n/m)(n/m-\|Z\|_2^2)}{(n/m)(\|W\|_2^2-\|Z\|_2^2)}.$$
Note that, conditioned on the event $\|W\|_2\geq \frac{\sqrt{n}}{\sqrt{m}
-c''\sqrt{k}}$ and $\|Z\|_2\leq \frac{\sqrt{n}}{\sqrt{m}+c''\sqrt{k}}$, we have
\begin{align*}
\frac{(\|W\|_2^2-n/m)(n/m-\|Z\|_2^2)}{(n/m)(\|W\|_2^2-\|Z\|_2^2)}
&\geq \frac{(\|W\|_2-\sqrt{n/m})(\sqrt{n/m}-\|Z\|_2)}{2\sqrt{n/m}(\|W\|_2-\|Z\|_2)}\\
&\geq\bar c\,\sqrt{\frac{k}{m}}\gg \varepsilon,
\end{align*}
as long as $c$ is chosen to be sufficiently small.
Hence, in view of \eqref{eq: aux404},
$$\Prob\big\{\dist\big((\sqrt{n/m}B_2^n)\cap B''\cap \widetilde E,B_2^2\big)\geq
1+4\varepsilon\big\}\geq 3/4.$$
Finally,
$$\Prob\big\{\dist\big(B\cap E,B_2^k\big)\geq
(1+4\varepsilon)(1-\varepsilon)/(1+\varepsilon)\big\}\geq 1/2,$$
and the result follows.
\end{proof}

\section{Remarks and open questions}

\begin{itemize}

\item A question of importance is whether the assertion of Theorem~\ref{th: random ell} holds without assuming
existence of a $1$-unconditional basis in $(\R^n,\|\cdot\|_B)$. This seems quite plausible,
although absence of a preferred orthogonal basis in this case suggests that Talagrand's $L_1\text{--}L_2$ bound
will likely be inapplicable. In our paper the phenomenon of superconcentration
is presented as a black box: we do not attempt to investigate
the matters that lie beneath the $L_1\text{--}L_2$ bound. Proving the assertion
of the theorem in full generality should require new tools.

\item The assumption of Theorem~\ref{th: random ell} that the convex body $B$ is in the $\ell$-position
is not something absolutely necessary. Rather, what we need is a sort of a ``balancing''
condition for the norm $\|\cdot\|_B$. In particular, it is natural to expect that the assertion of the main theorem
remains valid if the $\ell$-position is replaced with the one given by
$$\Exp\|G\|_B^q=1\quad \mbox{and}\quad
1=|\det \Id_n|=\sup\big\{|\det U|:\,U\in\R^{n\times n},\,\Exp\|G\|_{U^{-1}(B)}^q\leq 1\big\}$$
for some fixed $q\geq 1$.
One may further ask what are other {\it natural} positions in which the superconcentration
phenomenon guarantees better than quadratic dependence on $\varepsilon$ in the randomized Dvoretzky theorem.

\end{itemize}

\bigskip

{\bf Acknowledgments.} The research is partially supported by the Simons Foundation.
I would like to thank Ramon van Handel for an interesting discussion.
I am also grateful to Nicole Tomczak-Jaegermann, Alexander Litvak and Assaf Naor for valuable remarks.

\end{document}